\documentclass[11pt]{article}
\usepackage{amsthm, amsmath, amssymb, amsfonts, url, booktabs, tikz, setspace, fancyhdr, bm, mathrsfs}
\usepackage{hyperref}
\usepackage{geometry}
\usepackage{hyperref, enumerate}
\usepackage[shortlabels]{enumitem}
\usepackage[babel]{microtype}
\usepackage[english]{babel}
\usepackage[capitalise]{cleveref}
\usepackage{comment}
\usepackage{bbm}
\usepackage{csquotes}
\usepackage{mathabx}
\usepackage{tikz}
\usepackage{graphicx}
\usepackage{float}
\usepackage{amsmath}

\counterwithin{figure}{section}

\newtheorem{theorem}{Theorem}[section]

\newtheorem{lemma}[theorem]{Lemma}

\newtheorem{claim}[theorem]{Claim}

\theoremstyle{definition}

\newtheorem*{defn-non}{Definition}

\usepackage[linesnumbered, ruled]{algorithm2e}
\SetKwRepeat{Do}{do}{while}%

\newenvironment{poc}{\begin{proof}[Proof of the claim]}{\end{proof}}

\newcommand{\C}[1]{{\protect\mathcal{#1}}}

\usepackage{todonotes}
\newcommand{\A}{\mathcal{A}}
\newcommand{\Q}{\mathbb{Q}}
\newcommand{\Z}{\mathbb{Z}}
\newcommand{\F}{\mathbb{F}}

\newcommand{\rank}{\operatorname{rank}}
\newcommand{\ii}{\mathrm{i}}

\title{A superlogarithmic saving for Oddtown modulo composite numbers}
\author{
Yuhao Zhao\thanks{School of Mathematical Sciences, University of Science and Technology of China, Hefei, China. Email: \texttt{zhaoyh21@mail.ustc.edu.cn}.}
}
\date{}

\begin{document}

\maketitle

\begin{abstract}
Let $f_{\ell}(n)$ be the largest size of a family $\mathcal{A}\subseteq2^{[n]}$ such that no member has size divisible by $\ell$, while the intersection of every two distinct members has size divisible by $\ell$, and let $\omega(\ell)$ denote the number of distinct prime divisors of $\ell$. For any prime power $\ell$, the classical answer is $f_{\ell}(n)=n$. When $\omega(\ell)\geq 2$, Bukh, Chao, and Zheng recently proved $\omega(\ell)n-O_{\ell}\left(n^{\frac{\omega(\ell)-2}{\omega(\ell)-1}}(\log n)^{C_{\ell}}\right)\leq f_{\ell}(n)\leq\omega(\ell)n-2\omega(\ell)\log n+11$ for some $C_{\ell}>0$. When $\ell$ has at least two distinct odd prime divisors, they further used Fourier analysis to improve the upper bound to $f_{\ell}(n)\leq\omega(\ell)n-(2\omega(\ell)+\varepsilon_{\ell})\log n$ for some $\varepsilon_{\ell}>0$, provided that $n$ is sufficiently large in terms of $\ell$ . 

For every fixed $\ell$ with $\omega(\ell)\geq2$, we prove
\[
f_{\ell}(n)\leq\omega(\ell)n-\Omega_{\ell}(\log n\log\log n)
\]
for large $n$. The upper bound relies on a submatrix lemma of Bhowmick, Dvir, and Lovett, which is based on the bounded-torsion polynomial Freiman--Ruzsa conjecture recently proved by Gowers, Green, Manners, and Tao.
\end{abstract}

\section{Introduction}

For every positive integer \(n\), write \([n]=\{1,\ldots,n\}\). Throughout the paper, all logarithms are to base \(2\). A family \(\A\subseteq2^{[n]}\) is an \(\ell\)-\textit{Oddtown} if \(|A|\not\equiv0\pmod{\ell}\) for every \(A\in\A\), while \(|A\cap B|\equiv0\pmod{\ell}\) for all distinct \(A,B\in\A\). We denote the maximum possible size of such a family by \(f_{\ell}(n)\), and write \(\omega(\ell)\) for the number of distinct prime divisors of \(\ell\). %The notation \(O_{\ell}(g(n))\) means a quantity whose absolute value is at most a constant depending only on \(\ell\) times \(g(n)\); the notation \(\Omega_{\ell}(g(n))\) has the corresponding lower bound meaning. %Multiple subscripts indicate dependence on all listed parameters. %As usual, \(o(1)\) denotes a quantity tending to zero as \(n\) tends to infinity.

The classical Oddtown problem is the case \(\ell=2\). It originated in a question of Erd\H{o}s and was solved independently by Berlekamp \cite{Berlekamp1969} and Graver \cite{Graver1975}, and it was one of the most famous applications of linear algebra methods in combinatorics~\cite{BabaiFrankl2022}. The Oddtown theorem shows \(f_{2}(n)\le n\), which is tight, as witnessed by the family of all singletons. More generally, the same answer holds when the modulus is any prime power; see \cite[Ex.~1.1.23]{BabaiFrankl2022}. 

When $\omega(\ell)\ge 2$, the trivial bounds are \(n\leq f_{\ell}(n)\leq\omega(\ell)n\). An argument of Szegedy (see \cite[Ex.~1.1.28]{BabaiFrankl2022}) for \(\ell=6\) gives \(f_{6}(n)\leq2n-2\log n\) for \(n\neq3\). Bukh, Chao, and Zheng \cite{BukhChaoZheng2025} observed that the argument extends to square-free moduli, giving \(f_{\ell}(n)\leq\omega(\ell)n-\omega(\ell)\log n\), and subsequently proved 
\[
\omega(\ell)n-O_{\ell}\left(n^{\frac{\omega(\ell)-2}{\omega(\ell)-1}}(\log n)^{C_{\ell}}\right)\leq f_{\ell}(n)\leq\omega(\ell)n-2\omega(\ell)\log n+11
\]
for some $C_{\ell}>0$, whenever \(\omega(\ell)\geq2\). For moduli having at least two distinct odd prime divisors, their Fourier analytic argument gives \(f_{\ell}(n)\leq\omega(\ell)n-(2\omega(\ell)+\varepsilon_{\ell})\log n\) for some \(\varepsilon_{\ell}>0\), provided that \(n\) is sufficiently large in terms of \(\ell\) \cite{BukhChaoZheng2025}. 

In this note, we improve the deficit in the upper bound from logarithmic order to order \(\log n\log\log n\), as stated below.

\begin{theorem}\label{thm:main}
Let \(\ell\geq2\) be fixed with \(\omega(\ell)\geq2\). There exists some constant \(c_{\ell}>0\) such that
\[
f_{\ell}(n)\leq\omega(\ell)n-c_{\ell}\log n\log\log n
\]
for every sufficiently large \(n\).
\end{theorem}

Our proof relies on a submatrix lemma of Bhowmick, Dvir, and Lovett \cite{BhowmickDvirLovett2014}, which is based on the bounded-torsion polynomial Freiman-Ruzsa conjecture recently proved by Gowers, Green, Manners, and Tao \cite{GowersGreenMannersTao2026}.

\paragraph{Organization.} The rest of this note is organized as follows. We first collect some useful lemmas in Section \ref{sec:upper-lem}, including the submatrix lemma and some linear-algebraic preliminary lemmas. Next, we prove a cross-characteristic rank comparison lemma in Section \ref{sec:cross-rank}. Finally, we complete the proof of the upper bound in Section \ref{sec:upper-proof}.

\paragraph{Notation.} Fix an integer \(\ell\) with \(\omega(\ell)\geq2\) and write \(\ell=\prod_{s=1}^{\omega(\ell)}q_{s}\), where \(q_{s}=p_{s}^{\alpha_{s}}\) and the \(p_{s}\) are distinct primes. 
Over any commutative ring, the \emph{inner product} (or \emph{dot product}) of two vectors $u=(u_1,\dots,u_n)$ and $v=(v_1,\dots,v_n)$ is $u\cdot v:=\sum_{i=1}^n u_i v_i$. A matrix is called \emph{Boolean} if all its entries belong to \(\{0,1\}\). Every Boolean row \(x\in\{0,1\}^{n}\) naturally represents the set \(\{r:x_{r}=1\}\). Clearly, over $\mathbb{Q}$, the self inner product of a Boolean row is the size of this set, and the inner product of two Boolean rows is the size of their intersection. For row vectors \(z_{1},\ldots,z_{r}\), their \emph{Gram matrix} is \((z_{a}\cdot z_{b})_{a,b}\). Thus the entries of the Gram matrix of incidence rows record precisely the set sizes and intersections.

%\section{The upper bound}\label{sec:upper}
%In this section we prove the upper bound. 
For a field \(K\), the notation \(\rank_{K}M\) means the dimension over \(K\) of the row space of \(M\), which equals the dimension of its column space. %Reduction of an integer vector modulo \(p\) is always performed coordinate by coordinate. 
A square integer matrix is called \emph{rationally nonsingular} if it is invertible over \(\Q\), and a submatrix is \emph{constant} if all its entries are equal. A \emph{rectangle} \(M[R,C]\) is a submatrix obtained by selecting a set \(R\) of rows and a set \(C\) of columns. %We shall use rank over two different fields: congruences give information over \(\F_{p}\), whereas rational nonsingularity records ordinary linear independence over \(\Q\).

\section{Some useful lemmas}\label{sec:upper-lem}

In this section, let us first give some useful lemmas. 
We begin with the key submatrix lemma of Bhowmick, Dvir, and Lovett \cite{BhowmickDvirLovett2014} that will be used in the proof. A reader may take Lemma~\ref{lem:rectangle} below as the black box and skip directly to its statement; the next few paragraphs only show how the cited theorems imply exactly that formulation. 

A \textit{list} $A$ of size $t$ over a finite set $\Omega$ is an ordered $t$-tuple $A = (a_1,a_2,...,a_t)$, where each $a_i\in\Omega$. A list can have repetitions. If it does not, we say that it is \textit{twin-free}. For two lists $U=(u_1,\dots,u_t), V=(v_1,\dots,v_t)$ over $\mathbb{Z}_m^n$ let $P_{U,V}$ be the $t\times t$ matrix over $\mathbb{Z}_m$ defined by $P_{U,V} (i,j) =  u_i\cdot v_j$ for $1\le i, j\le t$. Let $\omega$ be a primitive root of unity of order $m$. The \textit{duality measure} of twin-free lists $A,B\subseteq \mathbb{Z}_m^n$ with respect to $\omega$ is defined as
\[
D_\omega(A,B)=\left| \mathbb{E}_{a\sim A, b\sim B} \left[\omega^{a\cdot b}\right] \right|.
\]
For a $t\times t$ matrix $M$ over $\mathbb{Z}_m$, let $\mathrm{rk}(M)$ denote the smallest $r$ such that $M= AB$, where $A$ is a $t\times r$ matrix over $\mathbb{Z}_m$ and $B$ is an $r\times t$ matrix over $\mathbb{Z}_m$. Clearly, when $m=p$ is a prime number, we have $\mathrm{rk}(M)=\mathrm{rank}_{\mathbb{F}_p}(M)$. The submatrix lemma of Bhowmick, Dvir, and Lovett \cite[Lemma~6.3]{BhowmickDvirLovett2014}, together with the bounded torsion polynomial Freiman--Ruzsa theorem of Gowers, Green, Manners, and Tao \cite{GowersGreenMannersTao2026}, gives the following consequence.

\begin{lemma}\label{lem:submatrix}
Let $s,m,n\ge 2$ be integers with $s\mid m$, and let $\omega$ be a primitive root of unity of order $s$. Suppose that $A,B$ are two twin-free lists of the same size over $\mathbb{Z}_s^n$ such that $D_\omega (A,B)\ge \frac{2}{3m^{3/2}}$, and let $\mathrm{rk}(P_{A,B}) = r \ge 2$. 
Then there exist lists $A' \subseteq A$, $B' \subseteq B$ such that $D_\omega(A',B') = 1$, where $|A'| \ge 2^{-c(m)r/\log r}|A|, |B'| \ge 2^{-c(m)r/\log r}|B|$ for some constant $c(m)>0$ depending only on $m$.
\end{lemma}

The submatrix lemma of Bhowmick, Dvir, and Lovett \cite[Lemma~6.3]{BhowmickDvirLovett2014} was originally conditional on the difference set form of the polynomial Freiman--Ruzsa conjecture. The Ruzsa inequality \(|A+A||A|^{2}\leq|A-A|^{3}\) converts \(|A-A|\leq K|A|\) into \(|A+A|\leq K^{3}|A|\). The theorem of Gowers, Green, Manners, and Tao \cite{GowersGreenMannersTao2026} then covers \(A\) by \(K^{O_{p}(1)}\) cosets of a subgroup of size at most \(|A|\), which is precisely the polynomial covering input used in the proof of the submatrix lemma.

Now we show how Lemma \ref{lem:submatrix} implies Lemma~\ref{lem:rectangle}. Let \(p\) be an odd prime, and let \(M\) be a \(t\times t\) Boolean matrix over $\mathbb{F}_p$ with distinct rows and distinct columns and \(d=\rank_{\F_{p}}M\ge 2\). Elementary linear algebra gives a full rank factorization \(M=UV^{\mathsf T}\) over \(\F_{p}\), where \(U\) and \(V\) each have \(d\) columns. Write \(u_{a}\) and \(v_{b}\) for their rows. 
%In the notation of the cited submatrix lemma, take both modulus parameters \(s\) and \(m\) equal to \(p\). 
Its inner product matrix \(P_{U,V}=(u_{a}\cdot v_{b})_{a,b}\) is exactly \(M\), so its rank parameter is \(d\). If two rows of \(U\) were equal, the corresponding rows of \(M\) would be equal, and the analogous statement holds for \(V\) and the columns of \(M\). Thus distinct rows and columns of \(M\) make both lists twin-free. Fix \(\zeta=e^{2\pi\ii/p}\), and let \(\theta\) be the proportion of entries of \(M\) equal to \(1\). The duality measure of \(U,V\), namely the absolute value of the average of \(\zeta^{u_{a}\cdot v_{b}}\), is
\[
\left|(1-\theta)+\theta\zeta\right|\geq\cos\left(\frac{\pi}{p}\right)\geq\frac{2}{3p^{3/2}},
\]
where the first inequality holds since the minimum distance from the origin to the chord joining \(1\) and \(\zeta\) is \(\cos(\frac{\pi}{p})\).  Applying Lemma \ref{lem:submatrix} with \(s=m=p\), we can find a submatrix with duality measure \(1\). Equality in the triangle inequality then forces all entries in the resulting rectangle to be equal. Hence the rectangle is constant as a Boolean matrix. This gives the following consequence.

\begin{lemma}\label{lem:rectangle}
For every fixed odd prime \(p\), there is a constant \(C_{p}>0\) %and a positive integer \(d_{p}\) 
such that the following holds. If \(M\) is a \(t\times t\) Boolean matrix with distinct rows and distinct columns and \(d=\rank_{\F_{p}}M\geq 2\), then \(M\) contains a constant submatrix \(M[R,C]\) satisfying
\[
|R|\geq2^{-C_{p}d/\log d}t\ \text{ and }\ |C|\geq2^{-C_{p}d/\log d}t.
\]
\end{lemma}

\subsection{Linear-algebraic lemmas} 
We next record a simple independence principle modulo a prime power.  
For an integer \(a\neq0\), we write \(v_{p}(a)\) for the largest integer \(\nu\) such that \(p^{\nu}\mid a\).

\begin{lemma}\label{lem:independence}
Let \(q=p^{\alpha}\), where \(p\) is prime, and let \(z_{1},\ldots,z_{r}\in\Z^{n}\). If \(z_{i}\cdot z_{j}\equiv0\pmod{q}\) for distinct \(i,j\), while \(z_{i}\cdot z_{i}\not\equiv0\pmod{q}\) for every \(i\), then \(z_{1},\ldots,z_{r}\) are linearly independent over \(\Q\).
\end{lemma}

\begin{proof}[Proof of Lemma~\ref{lem:independence}]
Suppose that \(\sum_{i=1}^{r}c_{i}z_{i}=0\) is a nonzero rational relation. Clear denominators, let \(\nu=\min_{i:c_{i}\neq0}v_{p}(c_{i})\), and divide all coefficients by \(p^{\nu}\). The coefficients remain integers, and some \(c_{j}\) is coprime to \(p\). Taking the inner product with \(z_{j}\) removes all terms with \(i\neq j\) modulo \(q\) and gives \(c_{j}z_{j}\cdot z_{j}\equiv0\pmod{q}\). An element is a unit in \(\Z/q\Z\) precisely when it is not divisible by \(p\). Thus \(c_{j}\) is invertible modulo \(q\), so the last congruence forces \(z_{j}\cdot z_{j}\equiv0\pmod{q}\), a contradiction.
\end{proof}

The following lemma has been proved in \cite{BukhChaoZheng2025}. 
We include its short proof for completeness. A family is \emph{totally isotropic modulo \(q\)} if every two vectors in the family, including a vector paired with itself, have inner product zero modulo \(q\). 

\begin{lemma}\label{lem:orthogonality}
Let \(q=p^{\alpha}\), where \(p\) is prime, and let \(0\leq k\leq n\) and \(s\geq0\) be integers. Suppose that \(x_{1},\ldots,x_{n-k},y_{1},\ldots,y_{s}\in\Z^{n}\) satisfy \(x_{i}\cdot x_{i}\not\equiv0\pmod{q}\) for every \(i\), \(x_{i}\cdot x_{j}\equiv0\pmod{q}\) for distinct \(i,j\), \(y_{a}\cdot y_{b}\equiv0\pmod{q}\) for all \(a,b\), and \(x_{i}\cdot y_{a}\equiv0\pmod{q}\) for all \(i,a\). Then
\[
\dim_{\F_{p}}\operatorname{span}_{\F_{p}}\{y_{1}\bmod p,\ldots,y_{s}\bmod p\}\leq\frac{k}{2}.
\]
\end{lemma}

\begin{proof}
Let \(m\) be the dimension on the left. If \(m=0\), there is nothing to prove, so assume \(m\geq1\). Retain \(y_{1},\ldots,y_{m}\) whose reductions modulo \(p\) form a basis, and let \(Y\) be the \(m\times n\) integer matrix with these retained vectors as rows.

\begin{claim}\label{clm:normalize-isotropic}
After replacing the retained \(y\)-vectors by suitable integer linear combinations and then by vectors congruent to them modulo \(q\), we may assume that their restriction to some \(m\)-element coordinate set \(S\) is exactly \(I_{m}\), without changing any required inner product modulo \(q\).
\end{claim}

\begin{poc}
Because the rows of \(Y\bmod p\) are independent, row reduction over the field \(\F_{p}\) supplies a set \(S\subseteq[n]\) of \(m\) pivot coordinates such that the restriction \(W\) of \(Y\) to \(S\) is nonsingular modulo \(p\). Thus \(\det W\) is not divisible by \(p\), so it is a unit modulo \(q=p^{\alpha}\), and \(W\) is invertible over \(\Z/q\Z\). Choose an integer matrix \(L\) representing its inverse, so \(LW\equiv I_{m}\pmod{q}\), and replace the retained rows by the rows of \(LY\). The matrix \(L\) is also invertible modulo \(p\), so their reductions still span the same \(m\)-dimensional space. Their restriction to \(S\) is \(I_{m}+qK\) for an integer matrix \(K\). We may alter these entries by vectors from \(q\Z^{n}\), supported on \(S\), so that the restriction becomes exactly \(I_{m}\). Integer linear combinations preserve the required congruences because every retained \(y\)-vector is orthogonal to every \(x\)-vector and every retained \(y\)-vector. Adding a vector from \(q\Z^{n}\) changes no inner product modulo \(q\). For ease of notation, we continue to call the transformed rows \(y_{1},\ldots,y_{m}\).
\end{poc}

Enumerate \(S=\{s_{1},\ldots,s_{m}\}\) according to the normalized identity block. For each \(x_{i}\), replace it by \(x_{i}-\sum_{a=1}^{m}(x_{i})_{s_{a}}y_{a}\). The resulting vector vanishes on \(S\), and all its inner products modulo \(q\) remain unchanged because the \(y\)-rows are totally isotropic and orthogonal to every \(x_{i}\). We again keep the same symbol \(x_{i}\) for the transformed vector. Delete the coordinates in \(S\), and denote the resulting vectors by \(x'_{i}\) and \(y'_{a}\). Since every transformed \(x_{i}\) is zero on \(S\), deleting these coordinates does not change any inner product involving an \(x'_{i}\). The restriction of \(y_{a}\) to \(S\) is the \(a\)-th standard basis vector, so deletion subtracts \(1\) from its self inner product and \(0\) from its inner product with every other retained vector. Therefore \(x'_{i}\cdot x'_{i}\not\equiv0\pmod{q}\), \(y'_{a}\cdot y'_{a}\equiv-1\pmod{q}\), and any two distinct vectors in the list \(x'_{1},\ldots,x'_{n-k},y'_{1},\ldots,y'_{m}\) are orthogonal modulo \(q\). Lemma~\ref{lem:independence} gives \(n-k+m\) independent vectors in \(\Q^{n-m}\). Thus \(n-k+m\leq n-m\), which is equivalent to \(m\leq\frac{k}{2}\).
\end{proof}

We will also use the following elementary block rank inequality. It says that if the upper left block has small rank, then the two off diagonal blocks cannot both carry almost all of the rank of the full matrix.

\begin{lemma}\label{lem:block-rank}
Over any field, if \(P=\begin{pmatrix}A&B\\C&D\end{pmatrix}\), then
\[
\rank B+\rank C\leq\rank P+\rank A.
\]
\end{lemma}

\begin{proof}
Put \(a=\rank A\). Using invertible row operations among the upper block rows and invertible column operations among the left block columns, first reduce \(A\) to a matrix with an \(I_{a}\) pivot block and zeros elsewhere. Further column operations clear the entries of \(B\) in the pivot rows, and further row operations clear the entries of \(C\) in the pivot columns. These operations preserve the rank of \(P\) and reduce it to a matrix of the form
\[
\begin{pmatrix}I_{a}&0&0\\0&0&B_{0}\\0&C_{0}&D_{0}\end{pmatrix}.
\]
%The pivot block contributes \(a\) to the rank. After deleting its rows and columns, choose \(\rank B_{0}\) independent upper rows; all of them vanish on the middle column block. Next choose \(\rank C_{0}\) bottom rows whose projections to the middle column block are independent. In any linear relation among these chosen rows, looking first at the middle block forces all coefficients of the bottom rows to be zero, and then looking at the right block forces all coefficients of the upper rows to be zero. These rows are therefore independent, so 
Then it is easy to see that \(\rank P\geq a+\rank B_{0}+\rank C_{0}\). During the elimination, at most \(a\) row directions were removed from \(B\), and at most \(a\) column directions were removed from \(C\). Hence \(\rank B\leq a+\rank B_{0}\) and \(\rank C\leq a+\rank C_{0}\). Combining these inequalities proves the lemma.
\end{proof}

\section{$\mathbb{Q}$-rank versus $\mathbb{F}_p$-rank}\label{sec:cross-rank}
In this subsection we will use the lemmas from the last subsection to estimate the \(\F_{p}\)-rank of a rational nonsingularity matrix. The proof uses the large rectangle recursion familiar from communication complexity, in particular from the work of Nisan and Wigderson \cite{NisanWigderson1995}. The idea is to find a large constant block, use the block rank inequality to make one off-diagonal block have roughly half the rank over \(\F_{p}\), and use rational nonsingularity to retain a large square minor inside that block. Repeating this rank reduction forces the original modular rank to be large.

\begin{lemma}\label{lem:cross-rank}
For every fixed odd prime \(p\), there are constants \(t_{p},\gamma_{p}>0\) such that every rationally nonsingular \(t\times t\) Boolean matrix \(P\) with \(t\geq t_{p}\) satisfies
\[
\rank_{\F_{p}}P\geq\gamma_{p}\log t\log\log t.
\]
\end{lemma}

\begin{proof}
For \(d\geq1\), let \(T_{p}(d)\) be the largest order of a rationally nonsingular square Boolean matrix of rank at most \(d\) over \(\F_{p}\), and let $M$ be such a matrix of order $T_p(d)$. Suppose the rank of $M$ over \(\F_{p}\) is \(e\leq d\). Note that rational nonsingularity over \(\Q\) forces both its rows and its columns to be distinct. Moreover, row reduction of a basis for its row space produces \(e\) pivot coordinates on which coordinate projection is injective on the entire row space. In particular, it is injective on the Boolean rows. Each projected row still belongs to \(\{0,1\}^{e}\), so there are at most \(2^{e}\) such rows. Thus the order is at most \(2^{e}\leq2^{d}\), and only finitely many orders are possible. Therefore \(T_{p}(d)\) is well defined and satisfies \(T_{p}(d)\leq2^{d}\).

\begin{claim}\label{clm:recurrence}
For all sufficiently large \(d\),
\[
T_{p}(d)\leq2^{1+C_{p}d/\log d}T_{p}\left(\left\lfloor\frac{d+1}{2}\right\rfloor\right).
\]
\end{claim}

\begin{poc}
Let \(M\) be a rationally nonsingular square Boolean matrix of order \(t=T_{p}(d)\), and put \(e=\rank_{\F_{p}}P\leq d\). All ranks in the rest of this claim, unless marked by \(\Q\), are over \(\F_{p}\). If \(e<d_{p}\) for some \(d_p\), then \(t\leq2^{e}\leq2^{d_{p}-1}\). Since \(T_{p}(\left\lfloor\frac{d+1}{2}\right\rfloor)\geq1\) and the prefactor in Claim~\ref{clm:recurrence} tends to infinity, the claim holds in this case after increasing the threshold for \(d\). 
We may increase \(d_{p}\) so that \(x/\log x\) is increasing for \(x\geq d_{p}\). Assume now that \(e\geq d_p\) is sufficiently large. 

By Lemma \ref{lem:rectangle}, we can find a constant $r\times s$ submatrix of \(M\) with at least \(\delta_{d}t\) rows and columns, where \(\delta_{d}=2^{-C_{p}d/\log d}\). After permuting rows and columns, write \(M=\begin{pmatrix}A&B\\C&D\end{pmatrix}\), where \(A\) is this constant \(r\times s\) block and \(r,s\geq\delta_{d}t\). The block \(A\) has rank at most \(1\), so Lemma~\ref{lem:block-rank} gives \(\rank B+\rank C\leq\rank M+\rank A\leq d+1\). Hence at least one of \(B,C\) has rank at most \(\left\lfloor\frac{d+1}{2}\right\rfloor\).

Note that the \(r\) rows of \((A\ B)\) are a subset of the rows of the rationally nonsingular matrix \(M\); hence they are independent over \(\Q\). Since \(A\) has rational rank at most \(1\), the inequality \(\rank_{\Q}(A\ B)\leq\rank_{\Q}A+\rank_{\Q}B\) gives \(\rank_{\Q}B\geq r-1\). Applying the same argument to the \(s\) selected columns gives \(\rank_{\Q}C\geq s-1\). It is clear that a matrix of \(\mathbb{Q}\)-rank \(h\) contains an \(h\times h\) rationally nonsingular submatrix. Therefore, either \(B\) is a block of \(\F_{p}\)-rank at most \(\left\lfloor\frac{d+1}{2}\right\rfloor\) and  contains a rationally nonsingular square submatrix of order \(r-1\), or \(C\) is a block of \(\F_{p}\)-rank at most \(\left\lfloor\frac{d+1}{2}\right\rfloor\) and contains one of order \(s-1\). In either case the square submatrix has order at least \(\min\{r-1,s-1\}\) and \(\F_{p}\)-rank at most \(\left\lfloor\frac{d+1}{2}\right\rfloor\). If \(\delta_{d}t\geq2\), its order is at least \(\frac{\delta_{d}t}{2}\), and therefore \(\frac{\delta_{d}t}{2}\leq T_{p}(\left\lfloor\frac{d+1}{2}\right\rfloor)\), which rearranges to the claimed recurrence. If \(\delta_{d}t<2\), then \(t<2\delta_{d}^{-1}\), and the same recurrence follows from \(T_{p}(\left\lfloor\frac{d+1}{2}\right\rfloor)\geq1\).
\end{poc}

\begin{claim}\label{clm:recurrence-solution}
There is a constant \(K_{p}>0\) such that \(T_{p}(d)\leq2^{K_{p}d/\log d}\) for every sufficiently large \(d\).
\end{claim}

\begin{poc}
Put \(\phi(d)=\frac{d}{\log d}\). The recurrence replaces \(d\) by approximately \(\frac{d}{2}\), while its multiplicative cost has logarithm \(1+C_{p}\phi(d)\). Since \(\phi(d/2)/\phi(d)\) tends to \(\frac{1}{2}\), choose \(D_{0}\) so large that Claim~\ref{clm:recurrence}, the inequality \(1\leq\phi(d)\), and the contraction \(\phi(\left\lfloor\frac{d+1}{2}\right\rfloor)\leq\frac{3}{5}\phi(d)\) all hold for \(d>D_{0}\). Since every \(T_{p}(d)\) is finite, we may choose \(K_{p}\) large enough that \(T_{p}(d)\leq2^{K_{p}\phi(d)}\) for \(2\leq d\leq D_{0}\) and \(C_{p}+1+\frac{3K_{p}}{5}\leq K_{p}\). For \(d>D_{0}\), take base \(2\) logarithms in Claim~\ref{clm:recurrence} and apply the induction hypothesis at \(\left\lfloor\frac{d+1}{2}\right\rfloor\). The resulting upper bound is \(1+C_{p}\phi(d)+\frac{3K_{p}}{5}\phi(d)\leq K_{p}\phi(d)\). This proves the claim by induction on \(d\).
\end{poc}

If a rationally nonsingular square Boolean matrix \(P\) has order \(t\) and rank \(d\) over \(\F_{p}\), then by the definition of \(T_{p}(d)\) we have \(t\leq T_{p}(d)\). The trivial bound \(T_p(d)\leq 2^d\) also gives \(d\geq\log t\). In particular, \(d\) is sufficiently large when \(t\) is sufficiently large, so Claim~\ref{clm:recurrence-solution} applies and gives \(\log t\leq\frac{K_{p}d}{\log d}\). Rearranging gives \(d\geq\frac{\log t\log d}{K_{p}}\). Since \(d\geq\log t\), we also have \(\log d\geq\log\log t\), and hence \(d\geq\frac{\log t\log\log t}{K_{p}}\). Taking \(\gamma_{p}=K_{p}^{-1}\) and increasing \(t_{p}\) completes the proof.
\end{proof}

\section{Proof of the main result}\label{sec:upper-proof}
Now we give the proof of Theorem~\ref{thm:main}.
\begin{proof}[Proof of Theorem~\ref{thm:main}]
Suppose \(\omega(\ell)\geq2\). Let \(\A\) be an \(\ell\)-Oddtown on \([n]\), and write \(\ell=\prod_{i=1}^{\omega(\ell)}q_{i}\), where \(q_{i}=p_{i}^{\alpha_{i}}\). The prime powers \(q_{i}\) are pairwise coprime, so an integer is divisible by \(\ell\) exactly when it is divisible by every \(q_{i}\). For a set \(A\subseteq[n]\), let \(\boldsymbol{1}_{A}\in\{0,1\}^{n}\) be its incidence vector, defined by \((\boldsymbol{1}_{A})_{r}=1\) if \(r\in A\) and \((\boldsymbol{1}_{A})_{r}=0\) otherwise. Thus \(\boldsymbol{1}_{A}\cdot\boldsymbol{1}_{B}=|A\cap B|\), so all Oddtown conditions become inner product congruences.

For each \(i\in[\omega(\ell)]\), define \[\A_{i}=\{A\in\A:|A|\not\equiv0\pmod{q_{i}}\}\] and \(\A'_{i}=\A\setminus\A_{i}\).
It is clear that \(|\A|\leq\sum_{i=1}^{\omega(\ell)}|\A_{i}|\). For a fixed \(i\), the incidence vectors of \(\A_{i}\) have nonzero self inner products modulo \(q_{i}\), while distinct vectors are orthogonal modulo \(q_{i}\). By Lemma~\ref{lem:independence}, all incidence vectors of \(\A_i\) are independent in \(\Q^{n}\), and hence \(|\A_{i}|\leq n\).

Since \(\omega(\ell)\geq2\), at least one prime divisor of \(\ell\) is odd. Fix \(i\) with \(p_{i}\) odd, and let \(M_{i}\) be the \(|\A'_{i}|\times n\) incidence matrix whose rows are the vectors \(\boldsymbol{1}_{A}\) for \(A\in\A'_{i}\).

We shall estimate the rank of this matrix in two opposite directions. The orthogonality modulo \(q_{i}\) will give an upper bound, while a different prime power component will produce a large rationally nonsingular square submatrix, and Lemma~\ref{lem:cross-rank} will then give a lower bound over \(\F_{p_{i}}\).

\begin{claim}\label{clm:upper-deficit}
One has
\begin{equation*}
\rank_{\F_{p_{i}}}M_{i}\leq\frac{n-|\A_{i}|}{2}.
\end{equation*}
\end{claim}

\begin{poc}
Apply Lemma~\ref{lem:orthogonality} with the incidence vectors of \(\A_{i}\) as the \(x\)-vectors and those of \(\A'_{i}\) as the \(y\)-vectors. The \(x\)-vectors have nonzero self inner products modulo \(q_{i}\). Every \(y\)-vector has zero self inner product modulo \(q_{i}\), and all inner products between distinct rows vanish modulo \(q_{i}\) because they vanish modulo \(\ell\). There are \(|\A_{i}|=n-k\) vectors of the first kind, so \(k=n-|\A_{i}|\). The span of the reductions of the \(y\)-vectors modulo \(p_{i}\) is exactly the row space of \(M_{i}\) over \(\F_{p_{i}}\). Thus the conclusion of the lemma gives the desired claim.
\end{poc}

If \(|\A'_{i}|<\frac{n}{2}\), then \(|\A|=|\A_{i}|+|\A'_{i}|<\frac{3n}{2}\). Since \(\omega(\ell)\geq2\), this is stronger than \(\omega(\ell)n-c_{\ell}\log n\log\log n\) for any fixed \(c_{\ell}>0\) and all sufficiently large \(n\). We may therefore assume \(|\A'_{i}|\geq\frac{n}{2}\). Note that every member of \(\A'_{i}\) has size divisible by \(q_{i}\) but not by \(\ell\), so some other prime power \(q_{j}\) fails to divide its size. Hence \(\A'_{i}\subseteq\bigcup_{j\neq i}\A_{j}\). By the pigeonhole principle, there is some \(j\neq i\) for which \(\mathcal{B}:=\A'_{i}\cap\A_{j}\) has size
\begin{equation}\label{eq:upper-large-subfamily}
t:=|\mathcal{B}|\geq\frac{|\A'_{i}|}{\omega(\ell)-1}\geq\frac{n}{2(\omega(\ell)-1)}.
\end{equation}

\begin{claim}\label{clm:upper-minor}
The matrix \(M_{i}\) contains a rationally nonsingular \(t\times t\) Boolean submatrix \(P\), where \(t\geq\frac{n}{2(\omega(\ell)-1)}\).
\end{claim}

\begin{poc}
The incidence vectors of \(\mathcal{B}\) are linearly independent over \(\Q\) by Lemma~\ref{lem:independence}, applied modulo \(q_{j}\). Hence the \(t\times n\) incidence matrix of \(\mathcal{B}\) has row rank \(t\) over \(\Q\). Choosing \(t\) pivot columns produces a \(t\times t\) submatrix \(P\) with nonzero rational determinant. %Notice the change of coefficient fields: \(q_{j}\) is used only to prove independence over \(\Q\), whereas the rank of \(P\) below is measured over the different field \(\F_{p_{i}}\). 
Since \(\mathcal{B}\subseteq\A'_{i}\), the matrix \(P\) is a submatrix of \(M_{i}\), and \eqref{eq:upper-large-subfamily} gives the asserted lower bound for \(t\).
\end{poc}

Note that the quantity \(t\) is at least a positive constant multiple of \(n\), where the constant depends only on \(\ell\). Since \(P\) is obtained from \(M_{i}\) by deleting rows and columns, its rank over \(\F_{p_{i}}\) is at most that of \(M_{i}\). Lemma~\ref{lem:cross-rank} now yields
\begin{equation}\label{eq:lower-rank}
    \rank_{\F_{p_{i}}}M_{i}\geq\gamma_{p_{i}}\log t\log\log t\geq c'_{\ell}\log n\log\log n
\end{equation}
for a constant \(c'_{\ell}>0\) and all sufficiently large \(n\); here we used \(t\geq\frac{n}{2(\omega(\ell)-1)}\), which gives \(\log t\geq\frac{1}{2}\log n\) and \(\log\log t\geq\frac{1}{2}\log\log n\) once \(n\) is large enough. Combining the lower rank bound \eqref{eq:lower-rank} with Claim \ref{clm:upper-deficit} then gives \[|\A_{i}|\leq n-2c'_{\ell}\log n\log\log n.\]
Recall that every other \(\A_j\) has size at most \(n\), and consequently we obtain
\[
|\A|\leq\sum_{h=1}^{\omega(\ell)}|\A_{h}|\leq\omega(\ell)n-2c'_{\ell}\log n\log\log n.
\]
This proves the upper bound with \(c_{\ell}=2c'_{\ell}\).
\end{proof}

\paragraph{Acknowledgement.} 
The author would like to thank Prof. Zixiang Xu for helpful discussions.

%\bibliographystyle{abbrv}
%\bibliography{bib}
\end{document}